\documentclass[11pt]{article}%
\usepackage[letterpaper, portrait, margin=1in, headheight=0pt]{geometry}
\usepackage{mathtools}
\usepackage{enumerate}
\usepackage{amsmath,amsthm,amssymb,amsfonts}
\usepackage{enumitem}
\usepackage{mathrsfs}
\usepackage{tikz-cd}
\usepackage{multicol}
\usepackage[title]{appendix}
\usetikzlibrary{quotes,angles}

\DeclareMathOperator{\Hom}{Hom}
\DeclareMathOperator{\Bil}{Bil}

\DeclareMathOperator{\Inf}{Inf}
\DeclareMathOperator{\Tra}{Tra}
\DeclareMathOperator{\Res}{Res}
\DeclareMathOperator{\ima}{Im}

\newcommand{\F}{\mathbb{F}}
\newcommand{\E}{\varepsilon}
\newcommand{\U}{\omega}

\newcommand{\vp}{\varphi}
\newcommand{\LL}{L}

\newcommand{\cZ}{\mathcal{Z}}
\newcommand{\cB}{\mathcal{B}}
\newcommand{\cH}{\mathcal{H}}

\newtheorem{thm}{Theorem}[section]
\newtheorem{lem}[thm]{Lemma}

\newtheorem{cor}[thm]{Corollary}

\theoremstyle{definition}

\theoremstyle{remark}

\title{Multipliers and Unicentral Leibniz Algebras}
\author{Erik Mainellis}
\date{}

\begin{document}

\maketitle

\begin{abstract}
    This paper details the Leibniz generalization of Lie-theoretic results from Peggy Batten's 1993 dissertation. We first show that the multiplier of a Leibniz algebra is characterized by its second cohomology group with coefficients in the field. We then establish criteria for when the center of a cover maps onto the center of the algebra. Along the way, we obtain a collection of exact sequences and a brief theory of unicentral Leibniz algebras.
\end{abstract}

\section{Introduction}
In \cite{batten}, Peggy Batten established Lie analogues of results concerning the multipliers and covers of groups. In the first chapter of \cite{batten}, she proves that covers of Lie algebras are unique, a result which deviates from the group case. Batten also characterizes the multiplier in terms of a free presentation. In \cite{rogers}, the author proves the Leibniz analogue of this first chapter.

The aim of the present paper is to generalize Chapters 3 and 4 in \cite{batten} to Leibniz algebras. Given a central ideal of a Leibniz algebra $L$ and a central $L$-module, we first construct a Hochschild-Serre type spectral sequence of low dimension. This sequence is used to characterize the multiplier $M(L)$ of $L$ in terms of the second cohomology group $\cH^2(L,\F)$. We then establish conditions for when the center of any cover maps onto the center of the algebra, i.e., for when $\U(Z(E)) = Z(L)$, where $E$ is a cover of $L$ and $\U:E\xrightarrow{} L$ is a surjective homomorphism. These conditions are the special case of a four-part equivalence theorem that highlights an extension-theoretic crossroads of unicentral algebras, multipliers and covers, free presentations, and the second cohomology group of Leibniz algebras.

\section{Preliminaries}
Let $\F$ be a field. A \textit{Leibniz algebra} $L$ is an $\F$-vector space equipped with a bilinear multiplication which satisfies the \textit{Leibniz identity} $x(yz) = (xy)z + y(xz)$ for all $x,y,z\in L$. Let $A$ and $B$ be Leibniz algebras. An \textit{extension} of $A$ by $B$ is a short exact sequence $0\xrightarrow{} A\xrightarrow{\sigma} L\xrightarrow{\pi} B\xrightarrow{} 0$ such that $\sigma$ and $\pi$ are homomorphisms and $L$ is a Leibniz algebra. One may assume that $\sigma$ is the identity map, and we make this assumption throughout. An extension is called \textit{central} if $A\subseteq Z(L)$. A \textit{section} of an extension is a linear map $\mu:B\xrightarrow{} L$ such that $\pi\circ \mu = \text{id}_B$.

For a Leibniz algebra $L$, a pair of Leibniz algebras $(K,M)$ is called a \textit{defining pair} for $\LL$ if $\LL\cong K/M$ and $M\subseteq Z(K)\cap K'$. We say such a pair is a \textit{maximal defining pair} if the dimension of $K$ is maximal. In this case, $K$ is called a \textit{cover} for $\LL$ and $M$ is called the \textit{multiplier} for $\LL$, denoted $M(\LL)$. It immediately follows that $M$ is a unique Lie algebra since it is abelian, justifying $M(L)$ as being \textit{the} multiplier of $L$. As in \cite{batten}, $C(L)$ is used to denote the set of all pairs $(J,\lambda)$ such that $\lambda:J\xrightarrow{} L$ is a surjective homomorphism and $\ker \lambda\subseteq J'\cap Z(J)$. An element $(T,\tau)\in C(L)$ is called a \textit{universal element} in $C(L)$ if, for any $(J,\lambda)\in C(L)$, there exists a homomorphism $\beta:T\xrightarrow{} J$ such that $\lambda\circ \beta = \tau$.

As shown in \cite{rogers}, many results from Chapter 1 of \cite{batten} carry over to the Leibniz setting with few significant differences. One natural change is to consider algebras of the form $FR + RF$ as a replacement to the usual Lie product algebra $[F,R]$, which ensures a two-sided ideal for Leibniz algebras. A pair of dimension bounds are also notably different from the Lie case, but still ensure that both $K$ and $M$ have finite dimension. We now state the Leibniz version of the culminating result from the first chapter of \cite{batten}, as proven in \cite{rogers}.

\begin{thm}\label{batten 1.12}
Let $L$ be a finite-dimensional Leibniz algebra and let $0\xrightarrow{} R\xrightarrow{} F\xrightarrow{} \LL\xrightarrow{} 0$ be a free presentation of $L$. Let \[B= \frac{R}{FR+RF} ~~~~~~ C = \frac{F}{FR+RF} ~~~~~~ D=\frac{F'\cap R}{FR+RF}\] Then:
\begin{enumerate}
    \item All covers of $\LL$ are isomorphic and have the form $C/E$ where $E$ is the complement to $D$ in $B$.
    \item The multiplier $M(L)$ of $L$ is $D \cong B/E$.
    \item The universal elements in $C(L)$ are the elements $(K,\lambda)$ where $K$ is a cover of $L$.
\end{enumerate} 
\end{thm}

\section{Cohomology}\label{batten 3}
Consider a central extension $0\xrightarrow{} A\xrightarrow{} L\xrightarrow{} B\xrightarrow{} 0$ and section $\mu:B\xrightarrow{} L$. Define a bilinear form $f:B\times B\xrightarrow{} A$ by $f(i,j) = \mu(i)\mu(j) - \mu(ij)$ for $i,j\in B$. By our work in \cite{mainellis}, $f$ is a 2-cocycle of Leibniz algebras, meaning that $f(i,jk) = f(ij,k) + f(j,ik)$ for all $i,j,k\in B$. Moreover, the image of $f$ falls in $A$ by exactness. Since we are only working with central extensions for this section, we drop the trivial $\vp$ and $\vp'$ maps of central factor systems and let $\cZ^2(B,A)$ denote the set of all 2-cocycles. As usual, $\cB^2(B,A)$ is used to denote the set of all 2-coboundaries, i.e. 2-cocycles $f$ such that $f(i,j) = -\E(ij)$ for some linear transformation $\E:B\xrightarrow{} A$. We next recall that any elements $f$ and $g$ in $\cZ^2(B,A)$ belong to equivalent extensions if and only if they differ by a linear map $\E:B\xrightarrow{} A$, i.e. $f(i,j) - g(i,j) = -\E(ij)$ for all $i,j\in B$. In this case, we say $f$ and $g$ differ by a 2-coboundary. Therefore, extensions of $A$ by $B$ are equivalent if and only if they give rise to the same element of $\cH^2(B,A) = \cZ^2(B,A)/\cB^2(B,A)$, the \textit{second cohomology group} of $B$ with coefficients in $A$. Finally, the work in \cite{mainellis} guarantees that each element $\overline{f}\in \cH^2(B,A)$ gives rise to a central extension $0\xrightarrow{} A\xrightarrow{} L\xrightarrow{} B\xrightarrow{} 0$ with section $\mu$ such that $f(i,j) = \mu(i)\mu(j) - \mu(ij)$.

\subsection{Hochschild-Serre Spectral Sequence}
Let $H$ be a central ideal of a Leibniz algebra $L$ and let $0\xrightarrow{} H\xrightarrow{} L\xrightarrow{\beta} \LL/H\xrightarrow{} 0$ be the natural central extension with section $\mu$ of $\beta$. Let $A$ be a central $L$-module. The following theorem concerns a five-term cohomological sequence that we refer to as the Hochschild-Serre spectral sequence of low dimensions.

\begin{thm}\label{HS}
The sequence \[0\xrightarrow{} \Hom(\LL/H,A) \xrightarrow{\Inf_1} \Hom(\LL,A)\xrightarrow{\Res} \Hom(H,A) \xrightarrow{\Tra} \cH^2(\LL/H,A) \xrightarrow{\Inf_2} \cH^2(\LL,A)\] is exact.
\end{thm}

Before proving exactness, we need to define the maps of this sequence and check that they make sense. The first \textit{inflation map} $\Inf_1:\Hom(\LL/H,A)\xrightarrow{} \Hom(\LL,A)$ is defined by $\Inf_1(\chi) = \chi\circ \beta$ for any homomorphism $\chi:\LL/H\xrightarrow{} A$. Next, the \textit{restriction mapping} $\Res:\Hom(\LL,A)\xrightarrow{} \Hom(H,A)$ is defined by $\Res(\pi) = \pi\circ \iota$ where $\iota:H\xrightarrow{} \LL$ is the inclusion map. It is readily verified that $\Inf_1$ and $\Res$ are well-defined and linear.

Third is the \textit{transgression map} $\Tra:\Hom(H,A)\xrightarrow{} \cH^2(\LL/H,A)$. Let $f:L/H\times L/H\xrightarrow{} H$ be defined by $f(\overline{x}, \overline{y}) = \mu(\overline{x})\mu(\overline{y}) - \mu(\overline{x}\overline{y})$ and consider $\chi\in \Hom(H,A)$. Then $\chi\circ f\in \cZ^2(\LL/H,A)$ since $\chi\circ f(\overline{x}, \overline{yz}) - \chi\circ f(\overline{xy}, \overline{z}) - \chi\circ f(\overline{y}, \overline{xz}) = \chi(0) = 0$ for all $x,y,z\in \LL$. If $\nu$ is another section of $\beta$, let $g(\overline{x},\overline{y}) = \nu(\overline{x})\nu(\overline{y}) - \nu(\overline{x}\overline{y})$. Then $f$ and $g$ are cohomologous in $\cH^2(\LL/H,H)$, which implies that there exists a linear transformation $\E:\LL/H\xrightarrow{} H$ such that $f(\overline{x}, \overline{y}) - g(\overline{x}, \overline{y}) = -\E(\overline{x}\overline{y})$. Clearly $\chi\circ \E:\LL/H\xrightarrow{} A$ is also a linear transformation, and therefore $\chi\circ f$ and $\chi\circ g$ are cohomologous in $\cH^2(\LL/H,A)$. Letting $\Tra(\chi) = \overline{\chi\circ f}$, we have shown that $\Tra$ is well-defined. It is straightforward to verify that $\Tra$ is linear.

Finally, let $\Inf_2:\cH^2(\LL/H,A) \xrightarrow{} \cH^2(\LL,A)$ be defined by $\Inf_2(f+ \cB^2(\LL/H,A)) = f' + \cB^2(\LL,A)$, where $f'(x,y) = f(\beta(x),\beta(y))$ for $x,y\in \LL$ and $f\in \cZ^2(\LL/H,A)$. It is straightforward to verify that $\Inf_2$ is linear. To check that $\Inf_2$ maps cocycles to cocycles, one computes \begin{align*}
    0 &= f(\beta(x),\beta(y)\beta(z)) - f(\beta(x)\beta(y),\beta(z)) - f(\beta(y),\beta(x)\beta(z)) \\ &= f'(x,yz) - f'(xy,z) - f'(y,xz)
\end{align*} for all $x,y,z\in \LL$ since $f$ is a 2-cocycle. Hence $f'\in \cZ^2(\LL,A)$. To check that $\Inf_2$ maps coboundaries to coboundaries, suppose $f\in \cB^2(\LL/H,A)$. Then there exists a linear transformation $\E:\LL/H\xrightarrow{} A$ such that $f(\overline{x},\overline{y}) = -\E(\overline{xy})$ for $x,y\in \LL$. Note that $\beta(x) = x+H = \overline{x}$ for any $x\in\LL$. Therefore $f'(x,y) = f(\beta(x),\beta(y)) = -\E\circ \beta(xy)$, yielding $f'\in \cB^2(\LL,A)$.

\begin{proof}
Once again, we are concerned with the central extension $0\xrightarrow{} H\xrightarrow{} \LL\xrightarrow{\beta} \LL/H\xrightarrow{} 0$, a section $\mu$ of $\beta$, and a central $\LL$-module $A$. One has $f\in \cZ^2(\LL/H,H)$ for $f(\overline{x}, \overline{y}) = \mu(\overline{x})\mu(\overline{y}) - \mu(\overline{x}\overline{y})$. To show exactness at $\Hom(\LL/H,A)$, it suffices to show that $\Inf_1$ is injective. Suppose $\Inf_1(\chi) = 0$ for $\chi\in \Hom(\LL/H,A)$. Then $\chi\circ \beta(x)=0$ for all $x\in \LL$, which means that $\chi=0$ since $\beta$ is surjective.

To prove exactness at $\Hom(L,A)$, first consider an element $\chi\in \Hom(\LL/H,A)$. One computes $\Res(\Inf_1(\chi)) = \Res(\chi\circ \beta) = \chi\circ \beta\circ \iota = 0$ since $\iota$ includes $H$ into $\LL$ and $\beta$ sends elements of $H$ to zero in $\LL/H$. Thus $\ima(\Inf_1)\subseteq \ker(\Res)$. Conversely, consider an element $\chi\in\ker(\Res)$. Then $\chi\circ \iota = 0$ implies that $H\subseteq\ker(\chi)$. By the fundamental theorem of homomorphisms, there exists $\hat{\chi}\in \Hom(\LL/H,A)$ such that $\hat{\chi}\circ \beta = \chi$. But $\Inf_1(\hat{\chi}) = \hat{\chi}\circ \beta = \chi$. Hence $\ker(\Res)\subseteq \ima(\Inf_1)$.

To show exactness at $\Hom(H,A)$, first consider a map $\chi\in\Hom(\LL,A)$. Then \begin{align*}
    \chi\circ f(\overline{x},\overline{y}) &= \chi\circ \mu(\overline{x})\chi\circ\mu(\overline{y}) - \chi\circ \mu(\overline{x}\overline{y}) \\ &= -\chi\circ\mu(\overline{xy})
\end{align*} by centrality, which implies that $\chi\circ f\in \cB^2(\LL/H,A)$. Thus $\Tra(\Res(\chi)) = \Tra(\chi\circ \iota) = \overline{\chi\circ \iota\circ f} = 0$ and so $\ima(\Res)\subseteq \ker(\Tra)$. Conversely, let $\theta\in \Hom(H,A)$ be such that $\Tra(\theta) = \overline{\theta\circ f} = 0$. Then $\theta\circ f \in \cB^2(\LL/H,A)$ which implies that there exists a linear transformation $\E:\LL/H\xrightarrow{} A$ such that $\theta\circ f(\overline{x}, \overline{y}) = -\E(\overline{xy})$. Let $x=\mu(\overline{x}) + h_x$ and $y=\mu(\overline{y}) +h_y$. Then $xy = \mu(\overline{xy}) + h_{xy} = \mu(\overline{x})\mu(\overline{y})$ implies that \begin{equation}\label{chi on H}
    \theta(h_{xy}) = \theta(\mu(\overline{x})\mu(\overline{y}) - \mu(\overline{x}\overline{y})) = \theta\circ f(\overline{x}, \overline{y}) = -\E(\overline{xy}).
\end{equation}
Now let $\sigma(x) = \theta(h_x) +\E(\overline{x})$. Since $\ima\sigma\subseteq A$, $\sigma(x)\sigma(y) = 0$ by centrality. By (\ref{chi on H}), $\sigma(xy) = \theta(h_{xy}) + \E(\overline{x}\overline{y}) = 0$.
Hence $\sigma\in\Hom(\LL,A)$ and $\sigma(h) = \theta(h) + \E(h+H) = \theta(h)$ for all $h\in H$, which means that $\Res(\sigma) = \theta$ and thus $\ker(\Tra)\subseteq \ima(\Res)$.

To show exactness at $\cH^2(\LL/H,A)$, first consider a map $\chi\in \Hom(H,A)$. Then $\Tra(\chi) = \overline{\chi\circ f}$ where, as before, $f(\overline{x},\overline{y}) = \mu(\overline{x})\mu(\overline{y}) - \mu(\overline{xy})$ and $\chi\circ f\in \cZ^2(\LL/H,A)$. By definition of $\Inf_2$, \[\Inf_2(\overline{\chi\circ f}) = \overline{(\chi\circ f)'}\] where $(\chi\circ f)'(x,y) = \chi\circ f(\overline{x},\overline{y})$. We want to show that $(\chi\circ f)'$ is a coboundary in $\cH^2(\LL,A)$. To this end, we once again consider $x= \mu(\overline{x})+ h_x$ and $y= \mu(\overline{y}) + h_y$ with product $xy = \mu(\overline{x})\mu(\overline{y}) = \mu(\overline{xy}) - h_{xy}$. Then $\chi\circ f(\overline{x},\overline{y}) = \chi(\mu(\overline{x})\mu(\overline{y}) - \mu(\overline{xy})) = \chi(h_{xy})$. Define $\E(x) = -\chi(h_x)$. Then $\E:\LL\xrightarrow{} A$ and is linear. One computes $\E(xy) = -\chi(h_{xy}) = -\chi\circ f(\overline{x},\overline{y}) = -(\chi\circ f)'(x,y)$ which implies that $(\chi\circ f)'\in \cB^2(\LL,A)$. Therefore $\overline{(\chi\circ f)'} = 0$ and we have $\ima(\Tra)\subseteq \ker(\Inf_2)$. Conversely, suppose $g\in \cZ^2(\LL/H,A)$ such that $\overline{g}\in \ker(\Inf_2)$. Then $g(\overline{x},\overline{y}) = g'(x,y) = -\E(xy)$ for some linear $\E:L\xrightarrow{} A$. Since $\E$ is linear, $\E\circ f\in \cZ^2(\LL/H,A)$. As before, let $x= \mu(\overline{x}) + h_x \in \LL$ with $xy = \mu(\overline{x})\mu(\overline{y})$ the product of two such elements. Then \begin{align*}
    g'(x,y) &= g(\overline{x},\overline{y}) \\&= -\E(\mu(\overline{x})\mu(\overline{y})) \\&= -\E\circ f(\overline{x},\overline{y}) - \E\circ\mu(\overline{xy})
\end{align*} where $\E\circ \mu:\LL/H\xrightarrow{} A$. Thus $\overline{g} = \overline{-\E\circ f} = -\Tra(\E)$ which implies that $\ker(\Inf_2)\subseteq \ima(\Tra)$.
\end{proof}

\subsection{Relation of Multipliers and Cohomology}
The objective of this section is to prove that the multiplier $M(L)$ of a finite-dimensional Leibniz algebra $L$ is isomorphic to $\cH^2(L,\F)$, where $\F$ is considered as a central $L$-module. The following results rely on our ability to invoke the Hochschild-Serre spectral sequence, and so the labor of Theorem \ref{HS} begins to pay off.

\begin{thm}\label{if tra surj}
Let $Z$ be a central ideal in $L$. Then $L'\cap Z$ is isomorphic to the image of $\Hom(Z,\F)$ under the transgression map. In particular, if $\Tra$ is surjective, then $L'\cap Z\cong \cH^2(L/Z,\F)$.
\end{thm}

\begin{proof}
Let $0\xrightarrow{}Z\xrightarrow{} L\xrightarrow{} L/Z\xrightarrow{} 0$ be the natural exact sequence for a central ideal $Z$ in $L$. Then the sequence $\Hom(L,\F)\xrightarrow{\Res} \Hom(Z,\F)\xrightarrow{\Tra} \cH^2(L/Z,\F)$ is exact by Theorem \ref{HS}. Let $J$ denote the set of all homomorphisms $\chi:Z\xrightarrow{} \F$ such that $\chi$ can be extended to an element of $\Hom(L,\F)$. Then $J$ is precisely the image of the restriction map in $\Hom(Z,\F)$, which is equal to the kernel of the transgression map by exactness. This means that $\Hom(Z,\F)/J\cong \ima(\Tra)$ and thus it suffices to show that $\Hom(Z,\F)/J\cong L'\cap Z$.

Consider the natural restriction homomorphism $\Res_2:\Hom(Z,\F)\xrightarrow{} \Hom(L'\cap Z,\F)$. Since $Z$ and $L'\cap Z$ are both abelian, $\Res_2$ is surjective and $\Hom(L'\cap Z,\F)$ is the dual space of $L'\cap Z$. Therefore \[\frac{\Hom(Z,\F)}{\ker(\Res_2)}\cong \Hom(L'\cap Z,\F) \cong L'\cap Z\] and it remains to show that $J\cong \ker(\Res_2)$. For one direction, consider an element $\chi\in J$ with extension $\hat{\chi}\in \Hom(L,\F)$. Then $L'\subseteq \ker \hat\chi$ since $\F$ is abelian, which implies that $L'\cap Z\subseteq \ker \chi$. Thus $\chi\in \ker(\Res_2)$ and we have $J\subseteq \ker(\Res_2)$. Conversely, let $\chi\in \ker(\Res_2)$. Then $\chi\in \Hom(Z,\F)$ is such that $L'\cap Z\subseteq \ker \chi$, which implies that $\chi$ induces a homomorphism \[\chi':\frac{Z}{L'\cap Z}\xrightarrow{} \F\] defined by $\chi'(z+(L'\cap Z)) = \chi(z)$. Since \[\frac{Z}{L'\cap Z}\cong \frac{Z+L'}{L'},\] there exists a homomorphism \[\chi'':\frac{Z+L'}{L'}\xrightarrow{} \F\] defined by $\chi''(z+L') = \chi'(z+(L'\cap Z))$. But $\chi''$ can be extended to a homomorphism $\chi''':L/L'\xrightarrow{} \F$ which is defined by $\chi'''(x+L') = \chi''(x+L')$ for all $x\in Z$. Since $L/L'$ is abelian, $\chi'''$ can be extended to a homomorphism $\hat\chi:L\xrightarrow{} \F$ which is defined by $\hat\chi(x) = \chi'''(x+L')$. Therefore $\chi\in J$ and the first statement holds. The second statement holds since $\Tra$ maps $\Hom(Z,\F)$ to $\cH^2(L/Z,\F)$.
\end{proof}

Let $L$ be a Leibniz algebra with free presentation $0\xrightarrow{} R\xrightarrow{} F\xrightarrow{\U} \LL\xrightarrow{} 0$. The induced sequence \[0\xrightarrow{} \frac{R}{FR+RF}\xrightarrow{} \frac{F}{FR+RF} \xrightarrow{} \LL\xrightarrow{}0\] is a central extension since $RF$ and $FR$ are both contained in $FR+RF$. It is not unique, but has the following property.

\begin{lem}\label{restriction to R}
Let $0\xrightarrow{} A\xrightarrow{} B\xrightarrow{\phi} C\xrightarrow{}0$ be a central extension and $\alpha:\LL\xrightarrow{} C$ be a homomorphism. Then there exists a homomorphism $\beta:F/(FR+RF) \xrightarrow{} B$ such that \[\begin{tikzcd}
0\arrow[r]& \frac{R}{FR+RF}\arrow[r] \arrow[d, "\gamma"] & \frac{F}{FR+RF} \arrow[r] \arrow[d,"\beta"] & \LL\arrow[r] \arrow[d, "\alpha"] &0 \\
0\arrow[r] &A \arrow[r] & B\arrow[r] &C\arrow[r] &0
\end{tikzcd}\] is commutative, where $\gamma$ is the restriction of $\beta$ to $R/(FR+RF)$.
\end{lem}

\begin{proof}
Since $F$ is free, there exists a homomorphism $\sigma:F\xrightarrow{} B$ such that \[\begin{tikzcd}
 F\arrow[r, "\U"] \arrow[d, "\sigma", swap] & \LL \arrow[d,"\alpha"] \\ B \arrow[r, "\phi"] & C \end{tikzcd}\] is commutative. Let $r\in R\subseteq F$. Then $\U(r) = 0$ since $\ker \U = R$. Therefore $0=\alpha\circ \U(r) = \phi\circ\sigma(r)$ and so $\sigma(R)\subseteq \ker \phi$. We want to show that $FR+RF\subseteq \ker \sigma$. If $x\in F$ and $r\in R$, then $\sigma(xr) = \sigma(x)\sigma(r) = 0$ and $\sigma(rx) = \sigma(r)\sigma(x) = 0$ since $\sigma(r)\in \ker \phi = A \subseteq Z(B)$. Now $\sigma$ induces a homomorphism $\beta:F/(FR+RF)\xrightarrow{} B$. The left diagram commutes since we may take $A\xrightarrow{} B$ to be the inclusion map.
\end{proof}

\begin{lem}\label{tra surj}
Let $0\xrightarrow{} R\xrightarrow{} F\xrightarrow{} \LL\xrightarrow{} 0$ be a free presentation of $\LL$ and let $A$ be a central $\LL$-module. Then the transgression map $\Tra:\Hom(R/(FR+RF),A)\xrightarrow{} \cH^2(\LL,A)$ associated with \[0\xrightarrow{} \frac{R}{FR+RF}\xrightarrow{} \frac{F}{FR+RF}\xrightarrow{\phi} \LL\xrightarrow{} 0\] is surjective.
\end{lem}

\begin{proof}
Consider $\overline{g}\in \cH^2(\LL,A)$ and let $0\xrightarrow{} A\xrightarrow{} E\xrightarrow{\vp} L\xrightarrow{} 0$ be an associated central extension. By Lemma \ref{restriction to R}, there exists a homomorphism $\theta$ such that \[\begin{tikzcd}
0\arrow[r]& \frac{R}{FR+RF}\arrow[r] \arrow[d, "\gamma"] & \frac{F}{FR+RF} \arrow[r, "\phi"] \arrow[d,"\theta"] & \LL\arrow[r] \arrow[d, "\text{id}"] &0 \\
0\arrow[r] &A \arrow[r] & E\arrow[r, "\vp"] &\LL\arrow[r] &0
\end{tikzcd}\] is commutative and $\gamma = \theta|_{R/(FR+RF)}$. Let $\mu$ be a section of $\phi$. Then $\vp\circ\theta\circ \mu = \phi\circ \mu = \text{id}_L$ and so $\theta\circ \mu$ is a section of $\vp$. Let $\lambda = \theta\circ \mu$ and define $\beta(x,y) = \lambda(x)\lambda(y) - \lambda(xy)$. Then $\beta\in \cZ^2(\LL,A)$ and $\beta$ is cohomologous with $g$ since they are associated with the same extension. One computes \begin{align*}
    \beta(x,y) &= \theta(\mu(x))\theta(\mu(y)) - \theta(\mu(xy)) \\ &= \theta(\mu(x)\mu(y) - \mu(xy))\\ &= \gamma(\mu(x)\mu(y) - \mu(xy)) \\ &= \gamma(f(x,y))
\end{align*} where $f(x,y) = \mu(x)\mu(y) - \mu(xy)$ and $\gamma = \theta|_{R/(FR+RF)}$. Thus $\Tra(\gamma) = \overline{\gamma\circ f} = \overline{\beta} = \overline{g}$.
\end{proof}

\begin{lem}\label{set lemma}
If $C\subseteq A$ and $C\subseteq B$, then $A/C\cap B/C = (A\cap B)/C$.
\end{lem}

\begin{proof}
Clearly $(A\cap B)/C\subseteq A/C\cap B/C$. Let $x\in A/C\cap B/C$. Then $x=a+c_1 = b+c_2$ for $a\in A$, $b\in B$, and $c_1,c_2\in C$. Since $C\subseteq B$, $a=b+c_2-c_1\in B$, which implies that $a\in A\cap B$. Then $x=a+c\in (A\cap B)/C$ and so $A/C\cap B/C \subseteq (A\cap B)/C$.
\end{proof}

\begin{thm}
Let $\LL$ be a Leibniz algebra over a field $\F$ and $0\xrightarrow{} R\xrightarrow{} F\xrightarrow{} \LL\xrightarrow{} 0$ be a free presentation of $\LL$. Then \[\cH^2(\LL,\F)\cong \frac{F'\cap R}{FR+RF}.\] In particular, if $\LL$ is finite-dimensional, then $M(\LL)\cong \cH^2(\LL,\F)$.
\end{thm}

\begin{proof}
Let $\overline{R} = R/(FR+RF)$ and $\overline{F} = F/(FR+RF)$. Then $0\xrightarrow{} \overline{R}\xrightarrow{} \overline{F}\xrightarrow{} \LL\xrightarrow{} 0$ is a central extension. By Lemma \ref{tra surj}, \[\Tra:\Hom(\overline{R},\F)\xrightarrow{} \cH^2(\LL,\F)\] is surjective. By Theorem \ref{if tra surj}, \[\overline{F}'\cap \overline{R} \cong \cH^2(\overline{F}/\overline{R},\F) \cong \cH^2(\LL,\F).\] By Lemma \ref{set lemma}, \[\overline{F}'\cap \overline{R} \cong \frac{F'}{FR+RF} \cap \frac{R}{FR+RF} = \frac{F'\cap R}{FR\cap RF}.\] Therefore, \[M(\LL) = \frac{F'\cap R}{FR+RF} \cong \cH^2(\LL,\F)\] by the characterization of $M(L)$ from Theorem \ref{batten 1.12}.
\end{proof}

Thus is the multiplier $M(L)$ characterized by $\cH^2(L,\F)$. We have now proven the main result of Batten's Chapter 3 for the Leibniz case. We conclude this section with the Leibniz analogue of a corollary which appears at the end of said chapter.

\begin{cor}
For any cover $E$ of $\LL$ and any subalgebra $A$ of $E$ satisfying
\begin{enumerate}
    \item $A\subseteq Z(E)\cap E'$,
    \item $A\cong M(\LL)$,
    \item $\LL\cong E/A$,
\end{enumerate}
the associated transgression map $\Tra:\Hom(A,\F)\xrightarrow{} M(\LL)$ is bijective.
\end{cor}

\begin{proof}
First note that $0\xrightarrow{} A\xrightarrow{} E\xrightarrow{} \LL\xrightarrow{} 0$ is a central extension of $\LL$. Invoking the Hochschild-Serre spectral sequence yields \[0\xrightarrow{} \Hom(\LL,\F)\xrightarrow{\Inf_1} \Hom(E,\F)\xrightarrow{\Res} \Hom(A,\F)\xrightarrow{\Tra} \cH^2(\LL,\F)\xrightarrow{\Inf_2} \cH^2(E,\F)\] with $\ima(\Res) = \ker(\Tra)$. Furthermore, any $\theta\in \Hom(E,\F)$ yields $\Res(\theta)\in \Hom(A,\F)$. Now let $a\in A\subseteq E'$. Then $a=e_1e_2$ for some $e_1, e_2\in E$ which implies that $\Res(\theta(a)) = \Res(\theta(e_1)\theta(e_2)) = \Res(0) = 0$. Thus $\ima(\Res) = 0$, making $\ker(\Tra) = 0$, and so $\Tra$ injective. Since $\Hom(A,\F)\cong A\cong M(\LL)$, $\Tra$ is bijective.
\end{proof}

\section{Unicentral Leibniz Algebras}\label{batten 4}
Let $L$ be a Leibniz algebra. The objective of this section is to develop criteria for when the center of any cover of $L$ maps onto the center of $L$. One of these criteria will take the form of $Z(L)\subseteq Z^*(L)$, where $Z^*(L)$ denotes the intersection of all images $\U(Z(E))$ such that $0\xrightarrow{} \ker \U\xrightarrow{} E\xrightarrow{\U} L\xrightarrow{} 0$ is a central extension of $L$. It is easy to see that $Z^*(L)\subseteq Z(L)$. We say a Leibniz algebra $L$ is \textit{unicentral} if $Z(L) = Z^*(L)$. To prove our result, we will establish conditions for a more general central ideal $Z$ in $L$ before specializing to $Z(L)$.

\subsection{Sequences}
We begin by extending our Hochschild-Serre sequence. Let $Z$ be a central ideal in $L$ and consider the natural central extension $0\xrightarrow{} Z\xrightarrow{} L\xrightarrow{} L/Z\xrightarrow{}0$. To define our $\delta$ map, consider a cocycle $f'\in \cZ^2(L,\F)$ and define two bilinear forms $f_1'':L/L'\times Z\xrightarrow{} \F$ and $f_2'':Z\times L/L'\xrightarrow{} \F$ by $f_1''(x+L',z) = f'(x,z)$ and $f_2''(z,x+L') = f'(z,x)$ for $x\in L$ and $z\in Z$. To check that they are well-defined, one computes \begin{align*}
    f_1''(xy+L',z) &= f'(xy,z) \\ &= f'(x,yz) - f'(y,xz) \\ &= 0
\end{align*} and \begin{align*}
    f_2''(z,xy+L') &= f'(z,xy) \\ &= f'(x,zy) - f(xz,y) \\ &= 0
\end{align*} since $z\in Z(L)$. Hence $(f_1'',f_2'')\in \Bil(L/L'\times Z,\F)\oplus \Bil(Z\times L/L',\F)\cong L/L'\otimes Z\oplus Z\otimes L/L'$. Now consider a coboundary $f'\in \cB^2(L,\F)$. By definition, there exists a linear map $\E:L\xrightarrow{} \F$ such that $f'(x,y) = -\E(xy)$. One computes $f_1''(x+L',z) = f'(x,z) = -\E(xz) = 0$ and $f_2''(z,x+L') = f'(z,x) = -\E(zx) = 0$ since $z\in Z(L)$. Hence, a map $\delta:f'+\cB^2(L,\F) \mapsto (f_1'',f_2'')$ is induced which is clearly linear since $f'$, $f_1''$, and $f_2''$ are all in vector spaces of bilinear forms and the latter two are defined by $f'$.

\begin{thm}\label{batten 4.1}
Let $Z$ be a central ideal of Leibniz algebra $L$. The sequence \[\cH^2(L/Z,\F)\xrightarrow{\Inf} \cH^2(L,\F)\xrightarrow{\delta} L/L'\otimes Z\oplus Z\otimes L/L'\] is exact.
\end{thm}

\begin{proof}
Let $f\in \cZ^2(L/Z,\F)$. Then $\Inf(f+\cB^2(L/Z,\F)) = f'+\cB^2(L,\F)$ where $f'$ is the cocycle defined by $f'(x,y) = f(x+Z,y+Z)$. We also have $\delta(f'+\cB^2(L,\F)) = (f_1'',f_2'')$ where, for all $x\in L$ and $z\in Z$, \begin{align*}
    f_1''(x+L',z) = f'(x,z) = f(x+Z,z+Z) = 0,\\
    f_2''(z,x+L') = f'(z,x) = f(z+Z,x+Z) = 0,
\end{align*} which implies that $\delta(\Inf(f+\cB^2(L/Z,\F))) = (f_1'',f_2'') = (0,0)$. Therefore $\ima(\Inf)\subseteq \ker \delta$.

Conversely, suppose $f'\in \cZ^2(L,\F)$ is such that $\delta(f'+\cB^2(L,\F)) = (f_1'',f_2'') = (0,0)$. Then, for all $x\in L$ and $z\in Z$, one has $0=f_1''(x+L',z) = f'(x,z)$ and $0=f_2''(z,x+L') = f'(z,x)$. Hence, for all $z,z'\in Z$ and $x,y\in L$, one computes \begin{align*}
    f'(x+z, y+z') &= f'(x,y) + f_1''(x+L',z') + f_2''(z,y+L') +f_1''(z+L',z') \\ &= f'(x,y),
\end{align*} which yields a bilinear form $g:L/Z\times L/Z\xrightarrow{} \F$, defined by $g(x+Z,y+Z) = f'(x,y)$, that is well-defined. Furthermore, $g\in \cZ^2(L/Z,\F)$ since $f'$ is a cocycle. Thus $\Inf(g+\cB^2(L/Z,\F)) = f'+\cB^2(L,\F)$ and so $\ker \delta\subseteq \ima(\Inf)$.
\end{proof}

\begin{thm}\label{ganea sequence}
(Ganea Sequence) Let $Z$ be a central ideal in a finite-dimensional Leibniz algebra $\LL$. Then the sequence \[L/L'\otimes Z \oplus Z\otimes \LL/\LL'\xrightarrow{} M(\LL)\xrightarrow{} M(\LL/Z)\xrightarrow{} \LL'\cap Z\xrightarrow{} 0\] is exact.
\end{thm}

\begin{proof}
Let $F$ be a free Leibniz algebra such that $\LL = F/R$ and $Z=T/R$ for some ideals $T$ and $R$ of $F$. Since $Z\subseteq Z(\LL)$, one has $T/R\subseteq Z(F/R)$ and $FT + TF \subseteq R$. Inclusion maps $\hat{\beta}:R\cap F'\xrightarrow{} T\cap F'$ and $\hat{\gamma}:T\cap F'\xrightarrow{} T\cap (F'+R)$ induce homomorphisms \[\frac{R\cap F'}{FR+RF}\xrightarrow{\beta} \frac{T\cap F'}{FT+TF}\xrightarrow{\gamma} \frac{T\cap (F'+R)}{R} \xrightarrow{} 0.\] Since $R\subseteq T$, one has \[\frac{T\cap (F'+R)}{R} = \frac{(T+R)\cap (F'+R)}{R} \cong \frac{(T\cap F') + R}{R}\] which implies that $\gamma$ is surjective. By Theorem \ref{batten 1.12}, \[M(\LL)\cong \frac{R\cap F'}{FR+RF} \hspace{.75cm} \text{ and }\hspace{.75cm} M(\LL/Z)\cong \frac{T\cap F'}{FT+TF}.\] Also \[\LL'\cap Z \cong (F/R)'\cap (T/R) \cong \frac{F'+R}{R}\cap \frac{T}{R} \cong \frac{(F'+R)\cap T}{R}.\] Therefore, the sequence $M(\LL/Z)\xrightarrow{\gamma} L'\cap Z\xrightarrow{} 0$ is exact. Since \[\ker \gamma = \frac{(T\cap F')\cap R}{FT+TF} = \frac{R\cap F'}{FT+TF} = \ima \beta,\] the sequence $M(L)\xrightarrow{\beta} M(L/Z)\xrightarrow{\gamma} L'\cap Z$ is exact.

It remains to show that $L/L'\otimes Z \oplus Z\otimes \LL/\LL'\xrightarrow{} M(\LL)\xrightarrow{\beta} M(\LL/Z)$ is exact. Define a pair of maps \begin{align*}
    \theta_1&:\frac{F}{R+F'}\times \frac{T}{R}\xrightarrow{} \frac{R\cap F'}{FR+RF}, \\
    \theta_2&:\frac{T}{R}\times \frac{F}{R+F'} \xrightarrow{} \frac{R\cap F'}{FR+RF}
\end{align*} by $\theta_1(f+(R+ F'),t+R) = ft+(FR+RF)$ and $\theta_2(t+R,f+(R+F')) = tf+(FR+RF)$. Both are bilinear because multiplication is bilinear. To check that $\theta_1$ and $\theta_2$ are well-defined, suppose $(f+(R+F'),t+R) = (f'+(R+F'),t'+R)$ for $t,t'\in T$ and $f,f'\in F$. Then $t-t'\in R$ and $f-f'\in R+F'$ which implies that $t=t'+r$ for $r\in R$ and $f=f'+x$ for $x\in R+F'$. One computes \begin{align*}
    tf - t'f' &= (t'+r)(f'+x) - t'f' \\ &= t'x + rf' + rx
\end{align*}
and \begin{align*}
    ft - f't' &= (f'+x)(t'+r) - f't' \\ &= xt' + f'r + xr
\end{align*} which both fall in $FR+RF$ by the Leibniz identity and the fact that $FT+TF\subseteq R$. Thus $\theta_1$ and $\theta_2$ are well-defined, and so induce linear maps \begin{align*}
    \overline{\theta_1}:\frac{F}{R+F'}\otimes \frac{T}{R}\xrightarrow{} \frac{R\cap F'}{FR+RF}, \\
    \overline{\theta_2}:\frac{T}{R}\otimes \frac{F}{R+F'}\xrightarrow{} \frac{R\cap F'}{FR+RF}.
\end{align*} These, in turn, yield a linear transformation \[\overline{\theta}:\frac{F}{R+F'}\otimes \frac{T}{R} \oplus \frac{T}{R}\otimes \frac{F}{R+F'}\xrightarrow{} \frac{R\cap F'}{FR+RF}\] defined by $\overline{\theta}(a,b) = \overline{\theta_1}(a) + \overline{\theta_2}(b)$. The image of $\overline{\theta}$ is \[\frac{FT+TF}{FR+RF}\] which is precisely equal to $\{x+(FR+RF)~|~ x\in R\cap F',~ x\in FT+TF\} = \ker \beta$. Thus the sequence \begin{alignat*}{2}
    \frac{F}{R+F'}\otimes \frac{T}{R} \oplus \frac{T}{R}\otimes \frac{F}{R+F'} \cong L/L'\otimes Z\oplus Z\otimes L/L' & \xrightarrow{} \frac{R\cap F'}{FR+RF} \cong M(L) \\
    &\xrightarrow{} \frac{F'\cap T}{FT+TF}\cong M(L/Z)
\end{alignat*} is exact.
\end{proof}

\begin{cor}
(Stallings Sequence) Let $Z$ be a central ideal of a Leibniz algebra $\LL$. Then the following map is exact: \[M(\LL)\xrightarrow{} M(\LL/Z)\xrightarrow{} Z\xrightarrow{} \LL/\LL'\xrightarrow{} \frac{L}{Z+\LL'}\xrightarrow{} 0.\]
\end{cor}

\begin{proof}
Let $F$ be a free Leibniz algebra such that $L=F/R$ and $Z=T/R$ for ideals $T$ and $R$ of $F$. Then $FT+TF \subseteq R$ since $Z\subseteq Z(L)$. The inclusion maps $R\cap F'\xrightarrow{} T\cap F'\xrightarrow{} T\xrightarrow{} F\xrightarrow{} F$ induce the following sequence of homomorphisms: \[\frac{R\cap F'}{FR+RF} \xrightarrow{\beta} \frac{T\cap F'}{FT+TF} \xrightarrow{\theta} \frac{T}{R}\xrightarrow{\alpha} \frac{F}{R+F'}\xrightarrow{\U} \frac{F}{T+F'}\xrightarrow{\gamma} 0\] To prove exactness for our desired sequence, we make use of the following facts:
\begin{enumerate}
    \item $M(L)\cong \frac{R\cap F'}{FR+RF}$,
    \item $M(L/Z)\cong \frac{T\cap F'}{FT+TF}$,
    \item $Z\cong T/R$,
    \item $\frac{F}{R+F'}\cong L/L'$,
    \item $\frac{F}{T+F'} = \frac{F}{T+F'+R} \cong \frac{(F/R)/(T+F'+R)}{R} \cong \frac{F/R}{T/R + (F'+R)/R} \cong \frac{L}{Z+L'}$.
\end{enumerate}
Thus do the following equalities suffice for this proof: \begin{enumerate}
    \item[i.] $\ker \theta = \{x+ (FT+TF)~|~ x\in T\cap F',~ x\in R\} = \frac{T\cap F'\cap R}{FT+TF} = \frac{R\cap F'}{FT+TF} = \ima \beta$,
    \item[ii.] $\ker \alpha = \{x+R~|~ x\in T,~x\in (R+F')\} = \frac{T\cap (R+F')}{R} = \frac{R+(T\cap F')}{R} = \ima \theta$,
    \item[iii.] $\ker \U = \{x +(R+F')~|~ x\in F,~ x\in (T+F')\} = \frac{F\cap (T+F')}{R+F'} = \frac{T+F'}{R+F'} = \ima \alpha$,
    \item[iv.] $\ker \gamma = \frac{F}{T+F'} = \ima \U$.
\end{enumerate}
\end{proof}

\subsection{The Main Result}
In the previous subsection, we defined maps $\delta$ and $\beta$ that appeared in the extended Hochschild-Serre and Ganea sequences respectively. The latter of these is called the natural map. The following statements, two of which involve these maps, make up the conditions of our four-part theorem.

\begin{enumerate}
    \item $\delta$ is the trivial map,
    \item $\beta$ is injective,
    \item $M(L)\cong \frac{M(L/Z)}{L'\cap Z}$,
    \item $Z\subseteq Z^*(L)$.
\end{enumerate}
The following pair of lemmas shows that the first three are equivalent.

\begin{lem}
Let $Z$ be a central ideal of finite-dimensional Leibniz algebra $L$ and let $\delta:M(L)\xrightarrow{} L/L'\otimes Z\oplus Z\otimes L/L'$ be as in Theorem \ref{batten 4.1}. Then \[M(L)\cong \frac{M(L/Z)}{L'\cap Z}\] if and only if $\delta$ is the trivial map. Here we have identitfied $L'\cap Z$ with its image in $M(L/Z)$.
\end{lem}

\begin{proof}
We invoke Theorems \ref{HS} and \ref{batten 4.1}, yielding an exact sequence \begin{alignat*}{2}
    0\xrightarrow{}\Hom(L/Z,\F)\xrightarrow{\Inf_1} \Hom(L,\F)\xrightarrow{\Res} \Hom(Z,\F)&\xrightarrow{\Tra} M(L/Z) \\ &\xrightarrow{\Inf_2} M(L)\xrightarrow{\delta} L/L'\otimes Z\oplus Z\otimes L/L'.
\end{alignat*} In one direction, suppose $\delta$ is the zero map. Then $M(L)\cong \ker \delta \cong \ima(\Inf_2)$. Since \[\ima(\Inf_2)\cong \frac{M(L/Z)}{\ker(\Inf_2)}\] and $\ker(\Inf_2) = \ima(\Tra)\cong L'\cap Z$ by Theorem \ref{if tra surj}, we have \[M(L)\cong \frac{M(L/Z)}{L'\cap Z}.\] Conversely, the isomorphism \[M(L)\cong \frac{M(L/Z)}{L'\cap Z} \cong \frac{M(L/Z)}{\ker(\Inf_2)} \cong \ima(\Inf_2) \cong \ker \delta\] implies that $\delta$ is trivial.
\end{proof}

\begin{lem}
Let $Z$ be a central ideal of a finite-dimensional Leibniz algebra $L$ and let $\beta:M(L)\xrightarrow{} M(L/Z)$ be as in Theorem \ref{ganea sequence}. Then \[M(L)\cong \frac{M(L/Z)}{L'\cap Z}\] if and only if $\beta$ is injective.
\end{lem}

\begin{proof}
By Theorem \ref{ganea sequence}, the sequence $M(L)\xrightarrow{\beta} M(L/Z)\xrightarrow{\alpha} L'\cap Z\xrightarrow{\U} 0$ is exact. Suppose $\beta$ is injective. Then $\ker \beta = 0$, which implies that \[M(L)\cong \ima \beta \cong \ker \alpha \cong \frac{M(L/Z)}{\ima \alpha} = \frac{M(L/Z)}{\ker \U} \cong \frac{M(L/Z)}{L'\cap Z}.\] Conversely, the isomorphism \[M(L)\cong \frac{M(L/Z)}{L'\cap Z} \cong \ima \beta\] implies that $\beta$ is injective.
\end{proof}

It remains to show that these conditions are equivalent to $Z\subseteq Z^*(L)$. This will lead to criteria for when $\U(Z(E)) = Z(L)$, where $E$ is any cover of $L$ and $0\xrightarrow{} \ker \U\xrightarrow{} E\xrightarrow{\U} L\xrightarrow{} 0$ is a central extension. Such an extension is called a \textit{stem extension}, i.e., a central extension $0\xrightarrow{} A\xrightarrow{} B\xrightarrow{} C\xrightarrow{} 0$ for which $A\subseteq B'$.

Consider a free presentation $0\xrightarrow{} R\xrightarrow{} F\xrightarrow{\pi} L\xrightarrow{} 0$ of $L$ and let $\overline{X}$ denote the quotient algebra $\frac{X}{FR+RF}$ for any $X$ such that $FR+RF\subseteq X\subseteq F$. Since $R=\ker \pi$ and $FR+RF\subseteq R$, $\pi$ induces a homomorphism $\overline{\pi}:\overline{F}\xrightarrow{} L$ such that the diagram \[\begin{tikzcd}
F\arrow[r,"\pi"]\arrow[d]& L\\
\overline{F} \arrow[ur, swap, "\overline{\pi}"]
\end{tikzcd}\] commutes. Since $\overline{R}\subseteq Z(\overline{F})$, there exists a complement $\frac{S}{FR+RF}$ to $\frac{R\cap F'}{FR+RF}$ in $\frac{R}{FR+RF}$ such that $S\subseteq R\subseteq \ker \pi$ and $\overline{S}\subseteq \overline{R}\subseteq \ker\overline{\pi}$. Thus $\overline{\pi}$ induces a homomorphism $\pi_S:F/S\xrightarrow{} L$ and a central extension $0\xrightarrow{} R/S\xrightarrow{} F/S\xrightarrow{\pi_S} L\xrightarrow{} 0$. This extension is stem since $R/S\cong \frac{R\cap F'}{FR+RF} = \ker \pi_S$ implies that $F/S$ is a cover of $L$.

\begin{lem}\label{center in center}
For every free presentation $0\xrightarrow{} R\xrightarrow{} F\xrightarrow{\pi} L\xrightarrow{} 0$ of $L$ and every central extension $0\xrightarrow{}\ker \U \xrightarrow{} E\xrightarrow{\U} L\xrightarrow{} 0$, one has $\overline{\pi}(Z(\overline{F}))\subseteq \U(Z(E))$.
\end{lem}

\begin{proof}
Since the identity map $\text{id}:L\xrightarrow{} L$ is a homomorphism, we can invoke Lemma \ref{restriction to R}, yielding a homomorphism $\beta:\overline{F}\xrightarrow{} E$ such that the diagram \[\begin{tikzcd}
0\arrow[r]& \frac{R}{FR+RF}\arrow[r] \arrow[d, "\gamma"] & \frac{F}{FR+RF} \arrow[r, "\overline{\pi}"] \arrow[d,"\beta"] & L\arrow[r] \arrow[d, "\text{id}"] &0 \\
0\arrow[r] &\ker \U \arrow[r] & E\arrow[r, "\U"] &L\arrow[r] &0
\end{tikzcd}\] is commutative (where $\gamma$ is the restriction of $\beta$ to $\overline{R}$).

Let $A=\ker \U$. Our first claim is that $E=A+\beta(\overline{F})$. Indeed, let $e\in E$. Then $\U(e) = \overline{\pi}(f)$ for some $f\in \overline{F}$, and so $\U(e)= \U\circ\beta(f)$ by diagram commutativity. This implies that $e-\beta(f)\in \ker \U = A$, meaning $e-\beta(f) =a$ for some $a\in A$. Thus $e=a+\beta(f)$.

Our second claim is that $\beta(Z(\overline{F}))$ centralizes both $A$ and $\beta(\overline{F})$. To see this, one first computes $\beta(Z(\overline{F}))\beta(\overline{F}) = \beta(Z(\overline{F})\overline{F}) = \beta(0) = 0$ and $\beta(\overline{F})\beta(Z(\overline{F})) = \beta(\overline{F}Z(\overline{F})) = \beta(0) = 0$. Next, we know that $AE$ and $EA$ are both zero, and so $A\beta(Z(\overline{F}))$ and $\beta(Z(\overline{F}))A$ are zero as well. But this implies that $\beta(Z(\overline{F}))$ centralizes $E$ by the first claim. Hence $\beta(Z(\overline{F}))\subseteq Z(E)$ and $\U\circ \beta(Z(\overline{F}))\subseteq \U(Z(E))$, which yields $\overline{\pi}(Z(\overline{F}))\subseteq \U(Z(E))$.
\end{proof}

\begin{thm}\label{batten 4.7}
For every free presentation $0\xrightarrow{} R\xrightarrow{} F\xrightarrow{\pi} L\xrightarrow{} 0$ of $L$ and every stem extension $0\xrightarrow{} \ker \U\xrightarrow{} E\xrightarrow{\U} L\xrightarrow{} 0$, one has $Z^*(L) = \overline{\pi}(Z(\overline{F})) = \U(Z(E))$.
\end{thm}

\begin{proof}
By Lemma \ref{center in center}, $\overline{\pi}(Z(\overline{F}))$ is contained in $\U'(Z(E'))$ for every central extension $0\xrightarrow{} \ker \U'\xrightarrow{} E'\xrightarrow{\U'} L\xrightarrow{} 0$ of $L$. Thus $\overline{\pi}(Z(\overline{F}))\subseteq \U(Z(E))$ for our stem extension. We also know that $Z^*(L)$ is the intersection of all images $\U'(Z(E'))$, and that $\overline{\pi}(Z(\overline{F}))$ is one of these images since $0\xrightarrow{} \overline{R}\xrightarrow{} \overline{F} \xrightarrow{\overline{\pi}} L\xrightarrow{} 0$ is central. Therefore $\overline{\pi}(Z(\overline{F})) = Z^*(L)$. Since this equality holds for all $F$, we can assume that $0\xrightarrow{} R/S\xrightarrow{} F/S\xrightarrow{\pi_S} L\xrightarrow{} 0$ is a stem extension where $S$ is defined as above. Since the cover $F/S$ is unique up to isomorphism, it now suffices to show that $\pi_S(Z(F/S)) = \overline{\pi}(Z(\overline{F}))$.

Let $T$ be the inverse image of $Z(F/S)$ in $F$ and consider the commutative diagram \[\begin{tikzcd}
F\arrow[r, "\pi_3"]\arrow[d,swap, "\pi_1"]& F/S\\
\overline{F} \arrow[ur, swap, "\pi_2"]
\end{tikzcd}\] where all mappings are the natural ones. Then $\overline{T} = \pi_1(T)$ by definition and $\pi_2(\overline{T}) = \pi_2\circ\pi_1(T) = \pi_3(T) = Z(F/S)$, yielding the diagram \[\begin{tikzcd}
T\arrow[r, "\pi_3"]\arrow[d,swap, "\pi_1"]& Z(F/S)\\
\overline{T} \arrow[ur, swap, "\pi_2"]
\end{tikzcd}\] where all maps denote their restrictions. Now let $x\in Z(\overline{F})$. Then $\pi_2(x)\in Z(F/S)$, which implies that there exists $y\in T$ such that $\pi_3(y) = \pi_2(x)$. The resulting equality $\pi_2\circ \pi_1(y) = \pi_2(x)$ yields an element $\pi_1(y)-x\in \ker \pi_2 = \overline{S} \subseteq \overline{T}$, where $\overline{S}\subseteq \overline{T}$ since $S\subseteq T$. Therefore $x\in \overline{T}$ and $Z(\overline{F})\subseteq \overline{T}$. For the reverse inclusion, we first note that $T/S = Z(F/S)$, and so $FT+TF \subseteq S$. Thus $\overline{F}\overline{T} + \overline{T}\overline{F}\subseteq \overline{S}$. Also $\overline{F}\overline{T} + \overline{T}\overline{F}\subseteq \overline{R}$ since $\overline{S}\subseteq \overline{R}$ and $\overline{F}\overline{T} + \overline{T}\overline{F}\subseteq \overline{F}'$ by definition. Hence $\overline{F}\overline{T} + \overline{T}\overline{F}\subseteq \overline{S}\cap (\overline{R}\cap\overline{F}') = \overline{0}$ which implies that $\overline{T} \subseteq Z(\overline{F})$ and thus $\overline{T} = Z(\overline{F})$. Thus, the commutative diagram \[\begin{tikzcd}
F\arrow[r, "\pi"]\arrow[d]& L\\
\overline{F} \arrow[ur, "\overline{\pi}"]\arrow[d] \\ F/S\ar[uur, swap, "\pi_S"]
\end{tikzcd}\] yields the equality $\overline{\pi}(Z(\overline{F})) = \overline{\pi}(\overline{T}) = \pi_S(T/S) = \pi_S(Z(F/S))$ by the definition of $T$.
\end{proof}

\begin{lem}
Let $Z$ be a central ideal of finite-dimensional Leibniz algebra $L$ and let $\beta:M(L)\xrightarrow{} M(L/Z)$ be as in Theorem \ref{ganea sequence}. Then $Z\subseteq Z^*(L)$ if and only if $\beta$ is injective.
\end{lem}

\begin{proof}
In the proof of the Ganea sequence, we saw that $\ker \beta$ can be interpreted as $\overline{F}\overline{T} + \overline{T}\overline{F}$. If $\beta$ is injective, then $\overline{F}\overline{T} + \overline{T}\overline{F} = 0$, which implies that $\overline{T}\subseteq Z(\overline{F})$. By the proof of Theorem \ref{batten 4.7}, $Z\subseteq Z^*(L)$. Conversely, if $Z\subseteq Z^*(L)$, then $\overline{T}\subseteq Z(\overline{F})$, which implies that $\overline{F}\overline{T} + \overline{T}\overline{F} = 0$. Thus $\ker \beta=0$ and $\beta$ is injective.
\end{proof}

\begin{thm}\label{batten 4.9}
Let $Z$ be a central ideal of a finite-dimensional Leibniz algebra $L$ and \[\delta:M(L)\xrightarrow{} L/L'\otimes Z\oplus Z\otimes L/L'\] be as in Theorem \ref{batten 4.1}. Then the following are equivalent:
\begin{enumerate}
    \item $\delta$ is the trivial map,
    \item the natural map $\beta$ is injective,
    \item $M(L)\cong \frac{M(L/Z)}{L'\cap Z}$,
    \item $Z\subseteq Z^*(L)$.
\end{enumerate}
\end{thm}

We conclude this section by narrowing our focus to when the conditions of Theorem \ref{batten 4.9} hold for $Z=Z(L)$.

\begin{thm}
Let $L$ be a Leibniz algebra and $Z(L)$ be the center of $L$. If $Z(L)\subseteq Z^*(L)$, then $\U(Z(E)) = Z(L)$ for every stem extension $0\xrightarrow{} \ker \U\xrightarrow{} E\xrightarrow{\U} L\xrightarrow{} 0$.
\end{thm}

\begin{proof}
By definition, $Z^*(L)\subseteq \U(Z(E)) \subseteq Z(L)$ for any stem extension $0\xrightarrow{} \ker \U\xrightarrow{} E\xrightarrow{\U} L\xrightarrow{} 0$. By hypothesis, $Z(L)\subseteq Z^*(L)$. Therefore $Z^*(L) = \U(Z(E)) = Z(L)$.
\end{proof}

\section*{Acknowledgements}
The author would like to thank Ernest Stitzinger for the many helpful discussions.

\end{document}